\documentclass[11pt]{article}
\usepackage{sustyle}
\usepackage[colorlinks=true,allcolors=blue,urlcolor={black}]{hyperref}
\usepackage[top=1.15in, bottom=1.3in, left=1.25in, right=1.25in]{geometry}

\usepackage{tikz}
\newcommand\ceil[1]{\lceil #1 \rceil}

\newcommand\pnull{p^{\textnormal{null}}}
\newcommand{\ep}[1]{\ifthenelse{\e\alphaual{#1}{1}}{\mathrm{e}}{\mathrm{e}^{#1}}}
\def\fdr{\textnormal{FDR}}
\def\fdp{\textnormal{FDP}}
\newcommand{\citet}[1]{\cite{#1}}
\newcommand{\citep}[1]{\cite{#1}}

\title{The FDR-Linking Theorem}
\author{Weijie J.~Su\thanks{Department of Statistics, the University of Pennsylvania. Email: \texttt{suw@wharton.upenn.edu}. Homepage: {\color{black}\url{http://stat.wharton.upenn.edu/\~suw/}.}}}
\date{}

\begin{document}
\maketitle

{\centering
\vspace*{-0.8cm}
\par\bigskip
\date\today\par
}

\begin{abstract}

This paper introduces the \texttt{FDR-linking} theorem, a novel technique for understanding \textit{non-asymptotic} FDR control of the Benjamini--Hochberg (BH) procedure under arbitrary dependence of the $p$-values. This theorem offers a principled and flexible approach to linking all $p$-values and the null $p$-values from the FDR control perspective, suggesting a profound implication that, to a large extent, the FDR of the BH procedure relies mostly on the null $p$-values. To illustrate the use of this theorem, we propose a new type of dependence only concerning the null $p$-values, which, while strictly \textit{relaxing} the state-of-the-art PRDS dependence (Benjamini and Yekutieli, 2001), ensures the FDR of the BH procedure below a level that is independent of the number of hypotheses. This level is, furthermore, shown to be optimal under this new dependence structure. Next, we present a concept referred to as \textit{FDR consistency} that is weaker but more amenable than FDR control, and the \texttt{FDR-linking} theorem shows that FDR
consistency is completely determined by the joint distribution of the null $p$-values, thereby reducing the analysis of this new concept to the global null case. Finally, this theorem is used to obtain a sharp FDR bound under arbitrary dependence, which improves the $\log$-correction FDR bound (Benjamini and Yekutieli, 2001) in certain regimes.

\end{abstract}

{\bf Keywords:} False discovery rate (FDR), Benjamini--Hochberg procedure, \texttt{FDR-linking} theorem, positive regression dependence within nulls (PRDN), FDR consistency, Simes method, compliance, Bonferroni-masked adversary, 

\section{Introduction}
\label{sec:introduction}

In 1995, Benjamini and Hochberg introduced a new procedure, henceforth referred to as the BH procedure, to control a type I error called the false discovery rate (FDR) in large-scale multiple testing problems \cite{BenjaminiH95}. To describe their procedure, order the $n$ observed $p$-values $p_1, \ldots, p_n$ from the most significant to the least as $p_{(1)} \le \cdots \le p_{(n)}$. Taking $0 < \alpha < 1$ as the nominal level throughout the paper, the BH procedure finds $R$ that is the last time $p_{(j)}$ is below the critical value $\alpha j/n$, that is,
\[
R = \max\left\{1 \le j \le n: p_{(j)} \le \alpha j/n \right\},
\]
with the convention $\max \emptyset = 0$. The BH procedure rejects the $R$ hypotheses with their corresponding $p$-values $p_j \le \alpha R/n$. Letting $V$ denote the number of falsely rejected null $p$-values\footnote{A $p$-value is null if its null hypothesis is true, in which case the $p$-value is stochastically larger than or equal to the uniform variable on $(0, 1)$.}, \cite{BenjaminiH95} proves that their procedure controls the FDR in the sense that
\[
\fdr := \E \left[ \frac{V}{\max\{R, 1\}} \right] \le \alpha
\]
under independence of the $p$-values\footnote{Formally, \cite{BenjaminiH95} requires that (1) the null $p$-values are jointly independent (2) and, furthermore, are independent of the non-null $p$-values. The bound can be strengthened to $\fdr \le \pi_0 \alpha$ \cite{BenjaminiH95}.}. 

After two decades of vigorous development, today the BH procedure and many of its variants have been extensively applied in high-throughput sciences such as genomics, where the FDR is increasingly recognized as the right error rate to control when simultaneously testing tens of thousands of, for example, gene expression levels (see \cite{tusher2001significance,genovese2002thresholding,storey2003statistical} for exemplary applications and \cite{benjamini2010discovering} for a review of the FDR concept). 

In stark contrast to the enormous popularity of the FDR criterion and methodologies, however, it is generally \textit{unclear} whether the BH procedure controls the FDR given a joint distribution of the $p$-values, though there are continued efforts to bridge this unsettling gap. To put it into perspective, the BH procedure has been proved to control the FDR under independence \cite{BenjaminiH95} or certain positive dependence \cite{prds} of the $p$-values. Surprisingly, \textit{no} dependence structure is known to yield \textit{meaningful} (non-asymptotic) FDR control while also being conceptually different from the two rather stringent dependence structures. That said, there is a long line of work concerning asymptotic FDR control of the BH
procedure under various
dependence assumptions (see Section~\ref{sec:related-work} for literature review).

This paper considers \textit{non-asymptotic} FDR control of the BH procedure under dependence of the $p$-values. Roughly speaking, our approach begins by decomposing the full dependence structure into the \textit{null dependence} (joint distribution of the null $p$-values) and the \textit{null-non-null dependence} (conditional distribution of the non-null $p$-values given the null $p$-values). One might imagine that each of the two distributions would have an arbitrarily large effect on the FDR. This is, surprisingly, \textit{not true}. As shown next, our main finding is that the FDR of the BH procedure admits an upper bound that is determined \textit{only} by the null dependence. 

To state our main theorem, let $\pi_0 := n_0/n$ denote the true null proportion, where $n_0$ is the number of null $p$-values. For any $0 < x < 1$, write $\fdr_0(x)$ for the false discovery rate of the BH procedure at level $x$ supplied with the null $p$-values\footnote{To clear up any confusion, we remark that the BH procedure applied to the nulls uses the critical values $x j/n_0$ for $j = 1, \ldots, n_0$.}. As such, the $\fdr_0(x)$ is entirely determined by the null dependence.
\begin{theorem}[The \texttt{FDR-linking} theorem]\label{thm:general}
Under arbitrary dependence of the $p$-values, the FDR of the BH procedure at level $\alpha$ satisfies
\begin{equation}\label{eq:fdr_link}
\fdr \le \pi_0\alpha + \pi_0 \alpha\int^{1}_{\pi_0 \alpha} \frac{\fdr_0(x)}{x^2} {\dx}x.
\end{equation}
\end{theorem}

This theorem makes a link from the nulls to the full set of the $p$-values from the FDR control standpoint, hence the name the \texttt{FDR-linking} theorem. As an attractive feature, Theorem~\ref{thm:general} holds \textit{unconditionally}, without requiring any distributional assumptions of the $p$-values. The FDR bound in the FDR-link inequality \eqref{eq:fdr_link} is a non-decreasing function of $\pi_0\alpha$ (see Proposition~\ref{prop:increase} in Section~\ref{sec:theor-refthm:r-refth}) and, therefore, a slightly weaker but simpler form of this result is
\[
\fdr \le \alpha +  \alpha\int^{1}_{ \alpha} \frac{\fdr_0(x)}{x^2} {\dx}x.
\]
As a useful fact, $\fdr_0(x)$ is just the type I error of the Simes method \cite{simes1986improved} at level $x$ on the nulls (see Lemma~\ref{lm:simes} in Section~\ref{sec:theor-refthm:r-refth}).

Importantly, Theorem~\ref{thm:general} implies that the null dependence is \textit{central} to FDR control of the BH procedure. To illustrate this point, recognize that while the FDR presumably depends on both the null dependence and the null-non-null dependence, the upper bound in \eqref{eq:fdr_link} is solely determined by the null dependence, no matter how \textit{adversarial/unfavorable} the dependence between the nulls and non-nulls is for FDR control. Put differently, the non-null $p$-values have a bounded effect on the FDR as opposed to the profound effect of the null $p$-values. Indeed, the nulls in the worst case ($\fdr_0(x) = 1$) would turn the FDR bound in \eqref{eq:fdr_link} into
\[
\pi_0\alpha + \pi_0 \alpha\int^{1}_{\pi_0 \alpha} \frac{1}{x^2} {\dx}x = 1,
\]
which is possibly the largest and least useful (see Section~\ref{sec:impr-log-corr} for the feasibility of $\fdr_0(x) = 1$).

The FDR bound in Theorem~\ref{thm:general} is \textit{optimal} for certain null dependence in the sense that \eqref{eq:fdr_link} can reduce to an equality (see Theorem~\ref{thm:optimal}). To obtain an upper bound on the FDR, therefore, the \texttt{FDR-linking} theorem suggests focusing on $\fdr_0(x)$ on the nulls, at least as a worthwhile starting
point. This flexibility is particularly useful in the case where the null dependence can be proven to yield a small value of $\fdr_0(x)$, which will be demonstrated through two applications of this theorem in Section~\ref{sec:whats-new} and Section~\ref{sec:appl-fdr-cons}, respectively.

For completeness, Theorem~\ref{thm:general} holds for all compliant procedures. In \cite{dwork2018differentially}, a multiple testing procedure is said to be compliant at level $\alpha$ if, denoting by $R$ the number of rejections, every rejected $p$-value $p_j$ satisfies
\[
p_j \le \frac{\alpha R}{n}.
\]
This condition is a special instance of the self-consistency condition first proposed in \cite{blanchard2008two}, where the threshold $\alpha R/n$ is replaced by a general function of $R$ and prior information about hypotheses. As is clear, compliant procedures include the step-down BH procedure, the generalized step-up-step-down procedures \cite{tamhane1998,sarkarstepwise} and, importantly, the (step-up, the one in Theorem~\ref{thm:general}) BH procedure, which is the most powerful in this family of procedures in the sense that the rejection set of any compliant procedure is a subset of that of the BH procedure. That said, although being the most powerful, the BH procedure in general does \textit{not} maximize the FDR among compliant procedures (see discussion in
Section~\ref{sec:prov-theor-refthm}). Interestingly, compliant procedures maintain the \texttt{FDR-linking} property as shown in Theorem~\ref{thm:general_compliance}, of which Theorem~\ref{thm:general} is a corollary.

\begin{theorem}[The \texttt{FDR-linking} theorem for compliant procedures]\label{thm:general_compliance}
Theorem~\ref{thm:general} holds for any compliant procedures.
\end{theorem}

To avoid any confusion, the FDR in \eqref{eq:fdr_link} in the context of Theorem~\ref{thm:general_compliance} corresponds to the compliant procedure at level $\alpha$, while the $\fdr_0(x)$ is for the BH procedure supplied with the nulls.

\subsection{Application: FDR control under a new condition}
\label{sec:whats-new}

As noted earlier, the BH procedure ensures $\fdr \le \alpha$ if the null $p$-values exhibit certain positive dependence and, furthermore, are positively dependent on the non-null $p$-values\cite{prds}. This type of dependence of the $p$-values is referred to as the \textit{positive regression dependence on a subset} property, or PRDS for short. To formally define this property, let $\mathcal{N}_0$ denote the set of nulls and call $D \subset [0, 1]^n$ an increasing set if $x \in D$ and $x \preceq y \preceq 1$ (in the coordinate-wise sense) together imply $y \in D$.

\begin{definition}[PRDS\footnote{In \cite{prds}, the PRDS property is defined for conditioning indices from any subset of all hypotheses. But the subset is almost always the nulls in the use of this property. Recognizing this fact, we assume the null set $\N_0$ in the definition for simplicity.}, \cite{prds}]\label{def:prds}
A set of $p$-values $(p_1, \ldots, p_n)$ is said to satisfy the PRDS property, if for any increasing set $D \subset [0, 1]^n$ and each null index $i \in \mathcal{N}_0$, the probability $\P((p_1, \ldots, p_n) \in D | p_i \le t)$ is non-decreasing in $t \in (0, 1]$.
\end{definition}

In relating to the literature, this property is implied by a condition termed the multivariate total positivity of order 2 ($\text{MTP}_2$) \cite{karlin1980classes}. The PRDS property is perhaps the most popular condition used for proof of FDR control of BH-type procedures (see, for example, the use of PRDS in super-uniformity lemmas for proving FDR control \cite{blanchard2008two,barber2017p,ramdas2017unified}). Although having relaxed the independence assumption in \cite{BenjaminiH95}, the PRDS remains to assume certain positive dependence between the nulls and non-nulls. This feature is \textit{undesirable} because, while some information of the null distribution is often available in applications, we typically lack knowledge of how the non-nulls depend on the nulls. 

Going back to the \texttt{FDR-linking} theorem, however, we learn that the FDR of any compliant procedure is largely contingent upon the null dependence. In this spirit, we ask if the PRDS property can be replaced by a more general condition. To address this need, we propose the following definition for FDR control, where, for convenience, a set of $p$-values is said to satisfy the PRD property if all $p$-values are null and obey the PRDS property.
\begin{definition}[PRDN]\label{def:prdn}
A set of $p$-values is said to satisfy the positive regression dependence within nulls (PRDN) if the set of all null $p$-values satisfies the PRD property.
\end{definition}

Put differently, the PRDN property boils down to the monotone non-decreasing of $\P(p_{\N_0} \in D_{\N_0} | p_i \le t)$ as a function $t \in (0, 1]$ for any $i \in \N_0$ and any increasing subset $D_{\N_0} \subset [0, 1]^{|\N_0|}$. Implied by the PRDS property, this new property imposes distributional assumptions \textit{only} on the null $p$-values. As such, the null-non-null dependence can be arbitrary in the PRDN property.

For example, consider an $n$-dimensional Gaussian vector $X \sim \N(\mu, \Sigma)$ with known $\Sigma$ for testing $\mu_i = 0$ against $\mu_i > 0, i = 1, \ldots, n$. The one-sided $p$-values for this testing problem obey the PRDN property if $\Sigma_{ij} \ge 0$ for all $i, j \in \N_0 = \{1 \le l \le n: \mu_l = 0\}$. In contrast, the PRDS property additionally requires the nonnegativity of $\Sigma_{ij}$ for all $i \in \N_0$ and $j \notin \N_0$ \cite{prds}. See the Appendix for a proof of this fact and an PRDN example of two-sided normal $p$-values \cite{karlin1981total}.

As shown below, any compliant procedure approximately controls the FDR for $p$-values satisfying this new property.
\begin{theorem}\label{thm:robustfdr}
Assume that the $p$-values satisfy the PRDN property. Then, any compliant procedure at level $\alpha$ obeys
\begin{equation}\label{eq:fdr2}
\fdr \le \alpha + \alpha\log\frac1{\alpha}.
\end{equation}
In particular, this FDR bound applies to the BH procedure.
\end{theorem}

The proof of this theorem follows from a simple application of the \texttt{FDR-linking} theorem in conjunction with the fact $\fdr_0(x) \le x$ (Lemma~\ref{lm:simes} in Section~\ref{sec:main-results}) under the PRDN property. In fact, the proof shows that Theorem~\ref{thm:robustfdr} can be strengthened to
\[
\fdr \le \pi_0\alpha + \pi_0\alpha\log\frac1{\pi_0\alpha},
\]
which implies \eqref{eq:fdr2} since $\alpha + \alpha\log\frac1{\alpha}$ is increasing in $\alpha \in (0, 1]$.

Theorem~\ref{thm:robustfdr} is the first result that maintains the FDR of the (original) BH procedure at a \textit{meaningful} level under a condition that is \textit{strictly less} stringent than PRDS. By ``meaningful,'' we mean that the FDR bound \eqref{eq:fdr2} can be set to 10\%, say, by taking the nominal level $\alpha = 2.05\%$ regardless of $n$, as opposed to the $\log$-correction bound \cite{prds}, which becomes too large to be useful as $n$ increases (see Section~\ref{sec:impr-log-corr}). The FDR bound \eqref{eq:fdr2} and particularly the logarithmic factor $\log\frac1\alpha$ \textit{cannot} be improved under the PRDN property, as seen from the lower bound below.

\begin{theorem}\label{thm:optimal}
Fix any $\epsilon > 0$. If $\alpha$ is sufficiently small (depending on $\epsilon$) and $n$ is sufficiently large (depending on both $\epsilon$ and $\alpha$),  then the BH procedure applied to certain PRDN $p$-values satisfies
\begin{equation}\nonumber
\fdr >  (1 - \epsilon) \left( \alpha + \alpha\log\frac1{\alpha} \right).
\end{equation}

\end{theorem}

The lower bound on the FDR is attained in an example where the null $p$-values are \iid~uniform random variables on $(0, 1)$, with certain least favorable non-null $p$-values constructed from an adversarial standpoint. 

%%Moreover, the true null proportion $\pi_0$ tends to 1 in this least-favorable example for the BH procedure. 

For more elaboration on this application of the \texttt{FDR-linking} theorem, see Section~\ref{sec:main-results}. This section also proves the \texttt{FDR-linking} theorem.

\subsection{Application: FDR consistency}
\label{sec:appl-fdr-cons}

To further explore the use of the \texttt{FDR-linking} theorem, we introduce \textit{FDR consistency}, which is a relaxation of FDR control in a certain sense. This new concept is inspired in part by Theorem~\ref{thm:robustfdr}, from which it follows that the FDR of the BH procedure working on any PRDN $p$-values converges to 0 as the nominal level $\alpha \goto 0$ by recognizing the fact
\[
\lim_{\alpha \goto 0} ~\alpha + \alpha \log\frac1{\alpha} = 0.
\]
The convergence is \textit{uniform}, regardless of the specific distribution of the PRDN $p$-values and, in particular, the number of $p$-values.

To formulate the definition of FDR consistency, let $\bm{\mathcal{P}}$ denote a class of $p$-value dependence structures. For example, $\bm{\mathcal{P}}$ can be all PRDN distributions on $[0, 1]^n$ for $n = 2, 4, 6, \ldots.$. Next, for any $p$-value distribution element $P \in \bm{\mathcal{P}}$, let $\fdr(\alpha; P)$ denote the FDR of the BH procedure at level $\alpha$ supplied with $p$-values sampled from the distribution $P$.

\begin{definition}[FDR consistency]\label{def:consist}
A dependence class $\bm{\mathcal{P}}$ is said to be FDR-consistent if
\[
\limsup_{\alpha \goto 0}\sup_{P \in \bm{\mathcal{P}}} \fdr(\alpha; P) = 0.
\]

\end{definition}

Our discussion above immediately reveals that the following result is true.
\begin{corollary}\label{cor:one_prdn}
All probability distributions of $p$-values satisfying the PRDN property form an FDR-consistent dependence class.
\end{corollary}

In words, instead of specifying the exact FDR level, FDR consistency only requires the FDR to tend to zero at a rate uniformly over all distribution elements in the class. In particular, FDR control at the nominal level (or up to a multiplicative constant) implies FDR consistency. This seemingly weak concept is, however, by no means trivial as there exists certain $p$-value distribution that violates FDR consistency (see Section~\ref{sec:impr-log-corr}). To appreciate this new concept, recognize that to ensure the FDR below a target level over an FDR-consistent class $\bm{\mathcal{P}}$, we only need to pick a certain nominal level for the BH procedure regardless of the number of $p$-values, albeit perhaps small. As such, the BH procedure controls the FDR at a ``meaningful'' level uniformly over the dependence class $\bm{\mathcal{P}}$.

In light of the \texttt{FDR-linking} theorem, a natural question is whether FDR consistency of the null dependence implies any useful properties of FDR control on all $p$-values. Remarkably, this is true: the FDR consistency is entirely a property of the null $p$-values. Loosely speaking, if the null dependence structures induced by a class $\bm{\mathcal{P}}$ form an FDR-consistent class, then the original class $\bm{\mathcal{P}}$ is ensured to be FDR-consistent. 

Now we formally state this result. For any $p$-value distribution $P$ in $[0, 1]^n$ for some $n$, let $P_0$ be the induced probability distribution of the null components of $P$ in $[0, 1]^{n_0}$ for some $n_0 \le n$. Moving on to a dependence class $\bm{\bm{\mathcal{P}}}$, we define the induced class
\[
\bm{\mathcal{P}}_0 := \{P_0: P \in \bm{\mathcal{P}}\},
\]
that is, the class of all null dependence induced by $\bm{\mathcal{P}}$.

\begin{theorem}\label{thm:consist}
If $\bm{\mathcal{P}}_0$ is FDR-consistent, then $\bm{\mathcal{P}}$ must also be FDR-consistent.

\end{theorem}

This theorem suggests that to prove FDR consistency, it is sufficient to only consider the null $p$-values. This profound implication significantly reduces the problem complexity. For comparison, FDR control in the strict sense does not enjoy such desired complexity reduction because, as noted in Theorem~\ref{thm:optimal}, a loss of a multiplicative factor is unavoidable in FDR control. Perhaps this is part of the reason why proving exact FDR control under general dependence is notoriously challenging. In light of the above, FDR consistency might serve as a target for future theoretical FDR research, at least being a reasonable starting point for proving exact FDR control.

In Section~\ref{sec:colt-style}, we prove Theorem~\ref{thm:consist} using the \texttt{FDR-linking} theorem and present several FDR-consistent dependence classes.

\subsection{Overview of other contributions}
\label{sec:peek-at-our}

In addition to the two aforementioned contributions, this paper applies the \texttt{FDR-linking} theorem to arbitrary dependence in Section~\ref{sec:impr-log-corr} and considers certain mildly adversarial non-null $p$-values for FDR control in Section~\ref{sec:bound-advers}. In short, Section~\ref{sec:impr-log-corr} uses the $\log$-correction bound under the global null \cite{prds} for $\fdr_0(x)$ in Theorem~\ref{thm:general_compliance}, which, together with the \texttt{FDR-linking} theorem, allows us to obtain state-of-the-art FDR guarantee under arbitrary dependence in certain regimes. In Section~\ref{sec:bound-advers}, we ask if the logarithmic factor in Theorem~\ref{thm:optimal} can be dropped by imposing some constraints on the null-non-null dependence and, interestingly, this is true if the non-null $p$-values are constructed without knowing the smallest null $p$-value.

\subsection{Related work}
\label{sec:related-work}

We give a brief review of the literature concerning FDR control of the BH procedure. In \cite{yekutieli2008false}, FDR control under non-PRDS dependence is examined through several examples in the non-asymptotic setting. In addition to introducing the PRDS property, \cite{prds} shows that the FDR is always bounded by $\alpha$ multiplied by a factor about $\log n$ under arbitrary dependence (the so-called $\log$-correction). This correction can be extended to incorporate prior knowledge such as weights of hypotheses via shape functions \cite{blanchard2008two,ramdas2017unified}. A dependence property called PRDSS is proposed in \cite{blanchard2008two}, which generalizes the PRDS property but is applied to a step-down procedure related to the Benjamini--Liu procedure~\cite{benjamini1999step}. In \cite{heesen2015}, FDR control is ensured under reverse martingale dependence, but this dependence is mainly motivated by the convenience of the martingale technique for proving FDR
control \cite{storey2004strong}. Recently, \cite{dwork2018differentially,katsevich2018towards} demonstrate that the FDR is maintained at a reasonable level as long as the nulls are independent.  

An integral part of the \texttt{FDR-linking} theorem is FDR control under the global null. This is equivalent to the validity of the Simes method. This method is known to control the type I error under many types of positive dependence structures \cite{hochberg1995extensions,samuel1996simes,sarkar1997simes,sarkar1998some}, for example, under the PRD property \cite{prds}. Convincing evidence from \cite{rodland2006simes,finner2017simes} suggests that the Simes method is likely to be approximately valid in a broader setting except for certain pathological dependence structures.

On a less related note, much more effort has been devoted to understanding asymptotic FDR control under dependence, using tools from empirical processes, factor models, and stochastic processes, among others \cite{yekutieli1999resampling,finner2001false,storey2004strong,genovese2004stochastic,finner2007dependency,farcomeni2007some,romano2008control,wu2008false,roquain2011exact,fan2012estimating,liu2017false}. Under (strong) dependence, the FDP is well-recognized to exhibit high variability \cite{owen2005variance,ferreira2006benjamini,efron2010correlated,schwartzman2011effect} and, accordingly, many methods have been developed to offer more reliable FDR control in the asymptotic setting \cite{efron2007correlation,sun2009large,friguet2009factor,bogdan2015slope,fan2017estimation,fan2018farmtest}.

%Test statistics \cite{clarke2009robustness,liu2014phase}
%%\cite{yekutieli1999resampling}
%%\cite{reiner2007fdr} for extensive simulation examples.

%%% Local Variables:
%%% mode: latex
%%% TeX-master: "paper"
%%% End:

\section{FDR Control Under PRDN}
\label{sec:main-results}

In this section, we detail the first application of the \texttt{FDR-linking} theorem introduced in Section~\ref{sec:whats-new}, proving Theorems~\ref{thm:robustfdr} and \ref{thm:optimal} in Sections~\ref{sec:theor-refthm:r-refth} and \ref{sec:prov-theor-refthm}, respectively. Notably, the treatment of Section~\ref{sec:theor-refthm:r-refth} is unusual in that we first prove the \texttt{FDR-linking} theorem and Theorem~\ref{thm:robustfdr} in fact follows as a corollary.

\subsection{Proving the \texttt{FDR-linking} theorem}
\label{sec:theor-refthm:r-refth}

We present an upper bound on the false discovery proportion (FDP)
\[
\frac{V}{\max\{R, 1\}}
\] 
for compliant procedures, including the BH procedure. Denote by $p_1^{\textnormal{null}}, \ldots, p_{n_0}^{\textnormal{null}}$ the null $p$-values and thus $V \le n_0$ records the number of rejections made from the $n_0$ null $p$-values\footnote{Throughout the paper, we assume the number $n_0$ of null hypotheses are nonzero. Otherwise, FDP is always 0 by definition.}. Let $p^{\textnormal{null}}_{(1)} \le p^{\textnormal{null}}_{(2)} \le \cdots \le p^{\textnormal{null}}_{(n_0)}$ be the ordered null $p$-values. As will be shown later, any compliant procedure must satisfy
\begin{equation}\label{eq:fdp_main0}
\fdp \le \min\left\{ \max_{1 \le j \le n_0} \frac{j}{\lceil n p^{\textnormal{null}}_{(j)}/\alpha \rceil}, 1 \right\},
\end{equation}
where the ceiling function $\ceil{x}$ maps $x$ to the least integer that is greater than or equal to $x$.

This bound on the FDP serves as the basis for proving Theorem~\ref{thm:general_compliance}. To verify this simple inequality, our approach is to take an adversarial viewpoint as this will shed light on the development of the remainder of the paper. Fixing the null $p$-values, we aim to maximize the FDP (consequently maximizing FDR) by seeking the least favorable scenario of the non-null $p$-values for FDR control. To this end, consider a thought experiment where an adversary knows which are the $n_0$ null $p$-values and can adversely alter the non-null $p$-values to maximize the FDP using a compliant procedure. By the definition of $V$, there must be one (falsely) rejected $p$-value that is no smaller than $p^{\textnormal{null}}_{(V)}$. This fact, together with the compliance property of the procedure, ensures that $p^{\textnormal{null}}_{(V)} \le \alpha R/n$ (recall that $R$ is the total number of rejections). Rearranging this inequality yields 
\[
R \ge \left\lceil n p^{\textnormal{null}}_{(V)}/\alpha \right\rceil.
\]
Therefore, if $V \ge 1$, we get 
\[
\fdp = \frac{V}{\max\{ R, 1 \}} = \frac{V}{R} \le \frac{V}{\lceil n p^{\textnormal{null}}_{(V)}/\alpha \rceil},
\]
from which \eqref{eq:fdp_main0} follows since $1 \le V \le n_0$ and FDP is always no greater than 1. As is clear, the maximum FDP can be achieved if the adversary sets $\lceil n p^{\textnormal{null}}_{(j^\star)}/\alpha \rceil - j^\star$ non-null $p$-values to 0 and the rest to 1 (suppose $n_1 = n-n_0$ is sufficiently large), where $j^\star$ is the index maximizing \eqref{eq:fdp_main0}. Note that \eqref{eq:fdp_main0} also holds for $V = 0$ because $\fdp = 0$ in this case. 

Next, we make a connection between \eqref{eq:fdp_main0} and the Simes method. Write
\begin{equation}\label{eq:alpha_bound_w}
\max_{1 \le j \le n_0} \frac{j}{\lceil n p^{\textnormal{null}}_{(j)}/\alpha \rceil} \le \max_{1 \le j \le n_0} \frac{\alpha j}{n p^{\textnormal{null}}_{(j)}} = \frac{\pi_0 \alpha}{\min_{1 \le j \le n_0} n_0 p^{\textnormal{null}}_{(j)}/j},
\end{equation}
where recall that $\pi_0 = n_0/n$. The random variable 
\begin{equation}\label{eq:simes_he}
\min_{1 \le j \le n_0} \frac{n_0 p^{\textnormal{null}}_{(j)}}{j}
\end{equation}
is precisely the Simes $p$-value restricted to the nulls. Let $F(x)$ denote its CDF. A closer look reveals that, for any $0 < x < 1$, $F(x)$ is precisely the false discovery rate of the BH procedure at level $x$ when restricted to the null $p$-values $p^{\textnormal{null}}_1, \ldots, p^{\textnormal{null}}_{n_0}$. In brief, when applied to the null $p$-values, the Simes method at level $x$ rejects the global null hypothesis if and only if the Simes $p$-value $\min_{1 \le j \le n_0} n_0 p^{\textnormal{null}}_{(j)}/j \le x$. In the case of rejection, it must hold that $p^{\textnormal{null}}_{(j)} \le x j/n_0$ for some $j$, meaning that the BH procedure applied to the nulls rejects at least one $p$-value, in which case the FDP is equal to 1. This relationship between the two methods ensures
\begin{equation}\label{eq:simes_cdf}
F(x) = \P\left( \min_{1 \le j \le n_0} n_0 p^{\textnormal{null}}_{(j)}/j \le x \right) = \fdr_0(x)
\end{equation}
for $0 < x < 1$. Taken together, the discussion above yields the following proof.

\begin{proof}[Proof of Theorem~\ref{thm:general_compliance}]
Making use of \eqref{eq:fdp_main0} and \eqref{eq:alpha_bound_w}, we get
\begin{equation}\label{eq:fdr_cdf}
\begin{aligned}
\fdr &\le \E \left[ \min\left\{ \frac{\pi_0 \alpha}{\min_{1 \le j \le n_0} n_0 p^{\textnormal{null}}_{(j)}/j}, 1 \right\} \right]\\
      &= \P\left(\min_{1 \le j \le n_0} n_0 p^{\textnormal{null}}_{(j)} /j \le \pi_0\alpha \right) + \E\left[\frac{\pi_0\alpha}{\min_{1 \le j \le n_0} n_0 p^{\textnormal{null}}_{(j)} /j};  \min_{1 \le j \le n_0} n_0 p^{\textnormal{null}}_{(j)} /j > \pi_0\alpha \right] \\
      &= F(\pi_0\alpha) + \int_{\pi_0\alpha}^1 \frac{\pi_0\alpha}{x} {\dx}F(x),
\end{aligned}
\end{equation}
where the integral is in the sense of the Riemann--Stieltjes integration by recognizing that the integrator $F$ is c\`adl\`ag. Using integration by parts, we proceed with the simplification of \eqref{eq:fdr_cdf} as
\[
\begin{aligned}
F(\pi_0\alpha) + \int_{\pi_0\alpha}^1 \frac{\pi_0\alpha}{x} {\dx}F(x) &= F(\pi_0\alpha) + \pi_0\alpha F(1) - F(\pi_0\alpha) + \pi_0\alpha\int^{1}_{\pi_0\alpha} \frac{F(x)}{x^2} {\dx}x\\
      &= \pi_0\alpha + \pi_0\alpha\int^{1}_{\pi_0\alpha} \frac{F(x)}{x^2} {\dx}x\\
      &= \pi_0\alpha + \pi_0\alpha\int^{1}_{\pi_0\alpha} \frac{\fdr_0(x)}{x^2} {\dx}x.
\end{aligned}
\]
This completes the proof for any compliance procedures, including the BH procedure.
\end{proof}

Next, we turn to Theorem~\ref{thm:robustfdr}. We have the following lemma concerning $p$-values that are PRDN distributed.
\begin{lemma}\label{lm:simes}
Assume that the $p$-values satisfy the PRDN property. Then, the Simes $p$-value \eqref{eq:simes_he} on the nulls is stochastically larger than or equal to the uniform random variable on $(0, 1)$.
\end{lemma}

This lemma is equivalent to
\[
\P\left(\min_{1 \le j \le n_0} n_0 p^{\textnormal{null}}_{(j)}/j \le x \right) \le x
\]
for $0 < x < 1$. This is true by using \eqref{eq:simes_cdf} and the fact $\fdr_0(x) \le x$ ensured by \cite{prds}. As an aside, in the case of \iid~uniform null $p$-values, $\min_{1 \le j \le n_0} n_0 p^{\textnormal{null}}_{(j)}/j$ is \textit{exactly} uniformly distributed on $(0, 1)$~\cite{simes1986improved}.

Applying Theorem~\ref{thm:general_compliance}, any compliant procedure with PRDN $p$-values gives
\[
\begin{aligned}
\fdr &\le \pi_0\alpha + \pi_0 \alpha \int^1_{\pi_0 \alpha} \frac{x}{x^2} {\dx}x \\
&= \pi_0 \alpha + \pi_0 \alpha \log\frac1{\pi_0 \alpha}\\
&\le \alpha + \alpha \log\frac1{\alpha}.
\end{aligned}
\]
This proves Theorem \ref{thm:robustfdr}. It should be emphasized that the analysis and reasoning above apply to not only the BH procedure but also any compliant procedures. More broadly, if the BH procedure controls the FDR up to a constant $c$ on the nulls or, equivalently, $\fdr_0(x) \le cx$, then the \texttt{FDR-linking} theorem concludes that
\[
\fdr \le \pi_0\alpha + c \pi_0\alpha\log\frac1{\pi_0\alpha}
\] 
for compliant procedures, roughly paying a factor of $\log\frac{1}{\alpha}$ in FDR control due to the least favorable non-nulls.

As additional background, we informally explain why the logarithmic factor $\log\frac1\alpha$ appears in the FDR bound in Theorem~\ref{thm:robustfdr}, though a rigorous treatment will be given in Section~\ref{sec:prov-theor-refthm}. Let all null $p$-values be \iid~uniform variables on $(0,1)$ and all non-null $p$-values be 0 and, additionally, assume $n - n_0$ is sufficiently large. Consider a procedure that rejects $R - 1$ non-null $p$-values and the smallest null $p$-value $p^{\textnormal{null}}_{(1)}$, where $R = \ceil{n p^{\textnormal{null}}_{(1)}/\alpha}$. This procedure is compliant. To see how the term $\alpha\log(1/\alpha)$ appears, first note that
\[
\fdp = \frac1{\ceil{n p^{\textnormal{null}}_{(1)}/\alpha}} \ge \frac{10}{11} \cdot \frac{\alpha}{n p^{\textnormal{null}}_{(1)}}
\]
in the case of $n p^{\textnormal{null}}_{(1)}/\alpha \ge 10$. If $n_0 \approx n$ and $n$ is large, $n p^{\textnormal{null}}_{(1)}$ is approximately distributed as an exponential random variable with mean 1. For sufficiently small $\alpha$, the logarithmic term $\alpha\log(1/\alpha)$ appears in
\begin{equation}\label{eq:adver_informal}
\begin{aligned}
\E\left[ \frac{10}{11} \cdot \frac{\alpha}{n p^{\textnormal{null}}_{(1)}};  n p^{\textnormal{null}}_{(1)}/\alpha \ge 10 \right] &\approx \E\left[ \frac{10}{11} \cdot \frac{\alpha}{\mathrm{Exp}(1)};  \mathrm{Exp}(1) \ge 10\alpha \right]\\
& = \frac{10\alpha}{11} \int_{10\alpha}^{\infty}\frac1{x}\e^{-x} \, {\dx}x\\
& = \frac{10\alpha}{11} \int_{10\alpha}^{\log\frac{12}{11}}\frac1{x}\e^{-x} \, {\dx}x + \frac{10\alpha}{11} \int_{\log\frac{12}{11}}^{\infty}\frac1{x}\e^{-x} \, {\dx}x\\
& \ge \frac{10\alpha}{11} \int_{10\alpha}^{\log\frac{12}{11}}\frac1{x} \frac{11}{12} \, {\dx}x + O(\alpha)\\
& = \frac{5}{6} \cdot \alpha\log\frac1{\alpha} + O(\alpha).
\end{aligned}
\end{equation}
The factor $5/6$ above can be made arbitrarily close to 1. As a crucial fact, the analysis above rests on knowing the most significant null $p$-value $p^{\textnormal{null}}_{(1)}$ (see Section~\ref{sec:bound-advers}).

In passing, we discuss an important property of the FDR bound in Theorem~\ref{thm:general}. Intuitively, one would expect that the FDR bound given in Theorem~\ref{thm:general} increases with the nominal level $\alpha$. This is indeed true as we show next.
\begin{proposition}\label{prop:increase}
The function
\[
 t  +  t \int^{1}_{ t } \frac{\fdr_0(x)}{x^2} {\dx}x
\]
is non-decreasing in $0 <  t  < 1$.
\end{proposition}

\begin{proof}[Proof of Proposition~\ref{prop:increase}]
Write
\[
f( t ) =  t  +  t \int^{1}_{ t } \frac{\fdr_0(x)}{x^2} {\dx}x.
\]
This function is continuous. Thus, it suffices to show that the right derivative of $f$ exists and is nonnegative everywhere. Interpreted as the CDF of the Simes $p$-value \eqref{eq:simes_he}, $\fdr_0(x)$ must be c\`adl\`ag. The c\`adl\`ag condition allows us to evaluate
\[
\begin{aligned}
\partial_+ f( t ) &= 1  + \int_{ t }^1 \frac{\fdr_0(x)}{x^2} {\dx}x - \frac{\fdr_0( t )}{ t }\\
& \ge 1  + \int_{ t }^1 \frac{\fdr_0( t )}{x^2} {\dx}x - \frac{\fdr_0( t )}{ t }\\
&= 1  + \fdr_0( t ) \left(\frac1{ t } - 1 \right) - \frac{\fdr_0( t )}{ t }\\
&= 1  - \fdr_0( t )\\
& \ge 0.
\end{aligned}
\]
For completeness, note that the first equality rests on the fact that $\fdr_0(t+) = \fdr_0(t)$, which is, again, ensured by the c\`adl\`ag condition. Hence, the proof is complete.

\end{proof}

\subsection{Optimality}
\label{sec:prov-theor-refthm}

In this section, we state and prove Theorem~\ref{thm:str}, a stronger version of Theorem~\ref{thm:optimal}. In what follows, let $\{\xi_j\}_{j=1}^{\infty}$ be \iid~exponential random variables with mean 1.
\begin{theorem}[A stronger version of Theorem~\ref{thm:optimal}]\label{thm:str}
For any $0 < \alpha < 1$, denote by
\[
D_{\alpha} := \E \left[ \min \left\{ \max_{1 \le j < \infty} \frac{j}{\lceil (\xi_1 + \cdots + \xi_j)/\alpha \rceil}, 1 \right\} \right].
\]
For any $\nu > 0$, there exist certain PRDN $p$-values on which the BH procedure gives
\begin{equation}\nonumber
\fdr >  D_{\alpha} - \nu.
\end{equation}
\end{theorem}

To be complete, the number $n$ of $p$-values in the example of this theorem should exceed some number depending on both $\alpha$ and $\nu$. In relating to Theorem~\ref{thm:robustfdr}, we immediately conclude that
\[
D_{\alpha} \le \alpha + \alpha \log\frac1\alpha
\]
for all $0 < \alpha < 1$. Alternatively, this inequality can be gleaned from the fact that
\[
\begin{aligned}
\E \left[ \min \left\{ \max_{1 \le j < \infty} \frac{j}{\lceil (\xi_1 + \cdots + \xi_j)/\alpha \rceil}, 1 \right\} \right] &\le \E \left[ \min \left\{ \max_{1 \le j < \infty} \frac{j}{(\xi_1 + \cdots + \xi_j)/\alpha}, 1 \right\} \right]\\
& = \E \left[ \min \left\{\frac{\alpha}{\min_{1 \le j < \infty} (\xi_1 + \cdots + \xi_j)/j }, 1 \right\} \right]\\
& = \alpha + \alpha\log\frac1\alpha.
\end{aligned}
\]
The last equality recognizes that the distribution of $\min_{1 \le j < \infty} (\xi_1 + \cdots + \xi_j)/j$ is uniform on $(0, 1)$.

The following lemma shows that for small $\alpha$, $D_{\alpha}$ is about as large as $\alpha + \alpha\log\frac1\alpha$. This explains why Theorem~\ref{thm:str} implies Theorem~\ref{thm:optimal} for sufficiently small $\alpha$.
\begin{lemma}\label{lm:c_alpha_b}
For any $\epsilon > 0$, if $\alpha$ is sufficiently small depending on $\epsilon$, then
\[
D_{\alpha} > (1 - \epsilon) \left( \alpha + \alpha\log\frac1{\alpha} \right).
\]
\end{lemma}

\begin{proof}[Proof of Theorem~\ref{thm:optimal}]
In Lemma~\ref{lm:c_alpha_b}, use $\epsilon/2$ in place of $\epsilon$. We have
\[
\fdr > D_{\alpha} - \nu > (1 - \epsilon/2) \left( \alpha + \alpha\log\frac1{\alpha} \right) - \nu
\]
if $\alpha$ is sufficiently small. Set $\nu$ in Theorem~\ref{thm:str} as $\nu = \frac{\epsilon}{2}\left( \alpha + \alpha\log\frac1{\alpha} \right)$. Thus, we get
\[
\fdr > (1 - \epsilon) \left( \alpha + \alpha\log\frac1{\alpha} \right),
\]
as desired.
\end{proof}

Next, we turn to the proof of Theorem~\ref{thm:str}. To construct certain PRDN $p$-values for Theorem~\ref{thm:str}, we start with the observation that the upper bound given in \eqref{eq:fdp_main0} is almost tight if the adversary is well-informed. More precisely, given that the adversary knows all null $p$-values, the adversary can find the index 
\begin{equation}\label{eq:j_st_def}
j^\star := \argmax_{1 \le j \le n_0} ~ \frac{j}{\ceil{n p^{\textnormal{null}}_{(j)}/\alpha}},
\end{equation}
where by convention $j^\star$ takes the largest value in the presence of ties. Having defined $j^\star$, we consider the following \textit{informed adversarial} non-null $p$-values: given $n_0$ null $p$-values, set
\[
\min\left\{ \left( \ceil{n p^{\textnormal{null}}_{(j^\star)}/\alpha} - j^\star \right)_+, ~ n_1 \right\}
\]
non-null $p$-values to 0 and the remaining non-null $p$-values to 1. Above, $x_+$ denotes $\max\{x, 0\}$. In the event that $\left( \ceil{n p^{\textnormal{null}}_{(j^\star)}/\alpha} - j^\star \right)_+ \le n_1$, we have the following fact concerning the BH procedure, where no distributional assumptions are made on the nulls.
\begin{lemma}\label{lm:rej_rule}
Supplied with the informed adversarial non-null $p$-values, the BH procedure at level $\alpha$ in the case where $\ceil{n p^{\textnormal{null}}_{(j^\star)}/\alpha} - j^\star \le n_1$ obeys
\begin{equation}\label{eq:1fdr_ceil1}
\fdp = \min\left\{ \max_{1 \le j \le n_0} \frac{j}{\lceil n p^{\textnormal{null}}_{(j)}/\alpha \rceil}, 1 \right\},
\end{equation}
thereby attaining the equality in \eqref{eq:fdp_main0} for the FDP.
\end{lemma}

This lemma confirms that the informed adversary indeed yields (nearly) the least favorable $p$-value configuration for the BH procedure in terms of FDR control. For completeness, if $1 \le \ceil{n p^{\textnormal{null}}_{(j^\star)}/\alpha} - j^\star \le n_1$, the BH procedure exactly rejects the $j^\star$ smallest null $p$-values and the $\ceil{n p^{\textnormal{null}}_{(j^\star)}/\alpha} - j^\star$ non-null $p$-values that are set to 0. On the flip side, if $n p^{\textnormal{null}}_{(j^\star)}/\alpha \le j^\star$, the BH procedure rejects at least the $j^\star$ smallest null $p$-values (possibly rejects more than $j^\star$ null $p$-values) but none of the non-null $p$-values, in which case the FDP is 1. This is in agreement with \eqref{eq:1fdr_ceil1}. Next is the proof of this lemma.

\begin{proof}[Proof of Lemma~\ref{lm:rej_rule}]
In the case where $1 \le \ceil{n p^{\textnormal{null}}_{(j^\star)}/\alpha} - j^\star \le n_1$, observe that
\[
\frac{\alpha (j^\star + \ceil{n p^{\textnormal{null}}_{(j^\star)}/\alpha} - j^\star)}{n} \ge \frac{\alpha n p^{\textnormal{null}}_{(j^\star)}/\alpha}{n} = p^{\textnormal{null}}_{(j^\star)},
\]
which implies that the BH procedure rejects at least the $j^\star$ smallest null $p$-values and the $\ceil{n p^{\textnormal{null}}_{(j^\star)}/\alpha} - j^\star$ zero-valued non-null $p$-values. Now, by way of contradiction, suppose, in addition to the $\ceil{n p^{\textnormal{null}}_{(j^\star)}/\alpha} - j^\star$ non-null $p$-values, the BH procedure rejects $j+l$ null $p$-values for some $l \ge 1$. By the compliance property of the BH procedure, it must hold that
\[
p^{\textnormal{null}}_{(j^\star + l)} \le \frac{\alpha (\ceil{n p^{\textnormal{null}}_{(j^\star)}/\alpha} + l)}{n},
\]
which implies
\[
\ceil{n p^{\textnormal{null}}_{(j^\star + l)}/\alpha} \le \ceil{n p^{\textnormal{null}}_{(j^\star)}/\alpha} + l.
\]
As such, it concludes that
\[
\frac{j^\star + l}{\ceil{n p^{\textnormal{null}}_{(j^\star + l)}/\alpha}} \ge \frac{j^\star + l}{\ceil{n p^{\textnormal{null}}_{(j^\star)}/\alpha} + l} > \frac{j^\star}{\ceil{n p^{\textnormal{null}}_{(j^\star)}/\alpha}},
\]
where the second inequality rests on the assumption that $j^\star < \ceil{n p^{\textnormal{null}}_{(j^\star)}/\alpha}$. This contradicts the definition of $j^\star$ in \eqref{eq:j_st_def}. Therefore, the BH procedure must reject exactly $j^\star$ null $p$-values and $\ceil{n p^{\textnormal{null}}_{(j^\star)}/\alpha} - j^\star$ non-null $p$-values, thus yielding
\[
\fdp = \frac{j^\star}{j^\star + \ceil{n p^{\textnormal{null}}_{(j^\star)}/\alpha} - j^\star} = \max_{1 \le j \le n_0} \frac{j}{\lceil n p^{\textnormal{null}}_{(j)}/\alpha \rceil} < 1.
\]

Next, we turn to the case where $n p^{\textnormal{null}}_{(j^\star)}/\alpha \le j^\star$. In this case, note that $p^{\textnormal{null}}_{(j^\star)} \le \alpha j^\star/n$. Thus, the BH procedure rejects at least the $j^\star$ most significant null $p$-values, and none of the non-null $p$-values is rejected because non-nulls are all set to 1.
\end{proof}

Taken together, the discussion above readily gives the proof of Theorem~\ref{thm:str}.
\begin{proof}[Proof of Theorem~\ref{thm:str}]
Lemma~\ref{lm:rej_rule} reveals that
\[
\fdp \ge \min\left\{ \max_{1 \le j \le n_0} \frac{j}{\lceil n p^{\textnormal{null}}_{(j)}/\alpha \rceil}, 1 \right\} - \bm{1}(\ceil{n p^{\textnormal{null}}_{(j^\star)}/\alpha} - j^\star > n_1)
\]
for the BH procedure supplied with the informed adversarial $p$-values, where $\bm{1}(\cdot)$ is the indicator function. Thus, we get\footnote{The analysis shows that the BH is not the most adversarial in terms of maximizing the FDP given the null $p$-values.}
\begin{equation}\label{eq:fdr_lower_b}
\fdr \ge \E \left[ \min \left\{ \max_{1 \le j \le n_0} \frac{j}{\lceil n  p^{\textnormal{null}}_{(j)}/\alpha \rceil}, 1 \right\} \right] - \P\left( \ceil{n p^{\textnormal{null}}_{(j^\star)}/\alpha} - j^\star > n_1 \right).
\end{equation}

The next step is to find certain PRD null $p$-values such that the right-hand side of \eqref{eq:fdr_lower_b} tends to $D_{\alpha}$ as $n \goto \infty$. To this end, we consider the case where the $n_0$ null-$p$-values are \iid~uniform variables $U_1, \ldots, U_{n_0}$ on $(0, 1)$. The asymptotic regime considered below satisfies $n_0, n_1 \goto \infty$ and $n_1/n_0 \goto 0$. Under these assumptions, we can prove that
\begin{equation}\label{eq:unif_rate}
\E \left[ \min \left\{ \max_{1 \le j \le n_0} \frac{j}{\lceil n U_{(j)}/\alpha \rceil}, 1 \right\} \right] \goto \E \left[ \min \left\{ \max_{1 \le j < \infty} \frac{j}{\lceil (\xi_1 + \cdots + \xi_j)/\alpha \rceil}, 1 \right\} \right]
\end{equation}
(recall $\xi_1, \xi_2, \ldots$ are \iid~exponential random variables with mean 1) and 
\begin{equation}\label{eq:unif_prob}
\P\left( \ceil{n U_{(j^\star)}/\alpha} - j^\star > n_1 \right) \goto 0,
\end{equation}
where $j^\star$ is the same as in \eqref{eq:j_st_def} using $U_{(j)}$ in place of $p^{\textnormal{null}}_{(j)}$. As is self-evident, \eqref{eq:fdr_lower_b} together with \eqref{eq:unif_rate} and \eqref{eq:unif_prob} finishes the proof of Theorem~\ref{thm:str}.

To complete the last step of the proof, we briefly explain why \eqref{eq:unif_rate} and \eqref{eq:unif_prob} are true. For the first inequality, note that $n U_{(j)}$ converges weakly to $\xi_1 + \cdots + \xi_j$ as $n \goto \infty$ since $n_0/n \goto 1$. This weak convergence holds simultaneously for $1 \le j \le J$ if $J$ is fixed. To bridge the gap between taking maximum over $\{1, \ldots, J\}$ and $\{1, \infty\}$, use the strong law of large numbers to show that both
\[
\max_{J+1 \le j \le n_0} \frac{j}{\lceil n U_{(j)}/\alpha \rceil} \text{ and } \max_{J+1 \le j < \infty} \frac{j}{\lceil (\xi_1 + \cdots + \xi_j)/\alpha \rceil}
\]
converge to $\alpha$ in probability as both $J, n_0 \goto \infty$. For the second inequality, the index $j^\star$ is bounded in probability as $n \goto \infty$ and, hence, one can slowly increase $n_1$ so that the probability of $\ceil{n U_{(j^\star)}/\alpha} - j^\star > n_1$ tends to zero.

\end{proof}

In passing, we briefly remark that, supplied with the informed adversarial $p$-values, the BH procedure is \textit{not} the most anti-conservative compliant procedure in the sense of maximizing the FDP. To see this, define
\[
j^\diamond := \argmax_{1 \le j \le n_0} \left\{ \frac{j}{\lceil n p^{\textnormal{null}}_{(j)}/\alpha \rceil}:  \lceil n p^{\textnormal{null}}_{(j)}/\alpha \rceil - j \le n_1 \right\},
\]
with the convention that $\argmax \emptyset = 0$ in the case where $\lceil n p^{\textnormal{null}}_{(j)}/\alpha \rceil - j > n_1$ for all $1 \le j \le n_0$. With this definition in place, the most anti-conservative compliant procedure rejects the following $p$-values, depending on whether $j^\diamond \ge 1$ or not,
\[
\begin{cases}
p^{\textnormal{null}}_{(1)}, \ldots, p^{\textnormal{null}}_{(j^\diamond)}, \text{ and }\left( \ceil{n p^{\textnormal{null}}_{(j^\diamond)}/\alpha} - j^\diamond \right)_+ ~ \text{zero-valued non-nulls} & \text{if } j^\diamond \ge 1\\
\emptyset & \text{if } j^\diamond = 0.
\end{cases}
\]
On a related note, if $\ceil{n p^{\textnormal{null}}_{(j^\star)}/\alpha} - j^\star \le n_1$, this procedure rejects the same set of $p$-values as the BH procedure.

%%% Local Variables:
%%% mode: latex
%%% TeX-master: "paper"
%%% End:

\section{FDR Consistency}
\label{sec:colt-style}

In this section, we prove Theorem~\ref{thm:consist} and showcase a number of FDR-consistent dependence classes. From Definition~\ref{def:consist}, FDR consistency can be equivalently defined as follows: a dependence class $\bm{\mathcal{P}}$ is FDR-consistent if the BH procedure ensures that, for all $p$-value distribution element $P \in \bm{\mathcal{P}}$,
\[
\fdr(\alpha; P) \le Q(\alpha)
\]
holds for a function $Q(\alpha)$ obeying $Q(\alpha) \goto 0$ as $\alpha \goto 0$. With this equivalent definition in place, we turn to the proof of Theorem~\ref{thm:consist}.

\begin{proof}[Proof of Theorem~\ref{thm:consist}]

By assumption, we know that
\[
\fdr(\alpha; P_0) \le Q(\alpha)
\]
for any $P_0 \in \bm{\mathcal{P}}_0$ and a function $Q(\alpha) \goto 0$ as $\alpha \goto 0$. Without loss of generality, assume $Q(\alpha) \le 1$. From the \texttt{FDR-linking} theorem, it immediately follows that
\begin{equation}\label{eq:ala_xdads}
\fdr(\alpha; P) \le \alpha + \alpha\int^{1}_{\alpha} \frac{Q(x)}{x^2} {\dx}x.
\end{equation}
Above, $P_0$ is the null dependence by $P$. 

Leveraging the fact $Q(\alpha) \goto 0$ as $\alpha \goto 0$, for any small $\epsilon > 0$, we can find $\delta > 0$ such that $Q(\alpha) < \epsilon$ if $\alpha < \delta$. By means of \eqref{eq:ala_xdads}, for $\alpha < \delta$, we get
\[
\begin{aligned}
\fdr(\alpha; P) &\le \alpha + \alpha\int^{\delta}_{\alpha} \frac{Q(x)}{x^2} \dx x + \alpha\int_{\delta}^1 \frac{Q(x)}{x^2} {\dx}x \\
          & \le \alpha + \alpha\int^{\delta}_{\alpha} \frac{\epsilon}{x^2} \dx x + \alpha\int_{\delta}^1 \frac{1}{x^2} {\dx}x \\
          & = \epsilon + \frac{(1-\epsilon)\alpha}{\delta}.
\end{aligned}
\]
Taking any $\alpha < \frac{\epsilon\delta}{1-\epsilon}$, we get $\fdr(\alpha; P) \le 2\epsilon$, which holds for all $P \in \bm{\mathcal{P}}$. Since $\epsilon$ can be arbitrarily small, what we have established reveals that 
\[
\lim_{\alpha\goto 0}\sup_{P \in \bm{\mathcal{P}}}\fdr(\alpha; P)  = 0.
\]
This completes the proof.

\end{proof}

Next, we present several examples of FDR-consistent dependence classes. Using Theorem~\ref{thm:consist}, it suffices to examine whether FDR consistency holds under the induced null dependence. For comparison, it is unclear whether the BH procedure controls the FDR at the nominal level or not in these examples, whereas FDR consistency is a much more feasible concept to deal with. That being said, a worthwhile point to make is that FDR consistency is no trivial concept in the sense that, for example, there exists a joint distribution of the $p$-values under the global null such that the BH procedure leads to \cite{guo2008control}
\[
\fdr = \min\left\{\left( 1 + \frac12 + \ldots + \frac1n \right) \alpha, 1 \right\} \approx \min\{ \alpha \log n, 1 \}
\]
for each $n$, which forms a non-FDR-consistent dependence class because the FDR bound involves $\log n$.

\subsection{FDR-consistent examples}
\label{sec:exampl-fdr-cons}

%%%%%%%%%%%%%%%%%%%%%%%%%%

\begin{example}[Equicorrelated normal distribution] \label{ex:equi_one_two} Consider a $p$-value vector whose null components are calculated from a (possibly negatively) equicorrelated normal distribution. More precisely, suppose we observe $X^0 \sim \N(0, \Sigma^0)$, where $\Sigma^0_{ij} = \rho$ if $i \ne j \le n_0$, and otherwise $\Sigma^0_{ii} = 1$. The correlation coefficient $\rho$ obeys $-\frac1{n_0 - 1} \le \rho < 0$. The null $p$-values are those for testing $\E X^0_i = 0$ against the one-sided or two-sided alternatives, while the non-null $p$-values are arbitrary. Then, the dependence class $\bm{\mathcal{P}}$ comprised of all such $p$-value distributions is FDR-consistent.

\end{example}

\begin{proof}
By Theorem~\ref{thm:consist}, we only need to consider the global null case $n = n_0$. For simplicity, the proof concerns the case of one-sided $p$-values and the proof for the two-sided case is very similar. Furthermore, Corollary~\ref{cor:one_prdn} shows that it is sufficient to only focus on $-\frac1{n-1} \le \rho < 0$, the negative correlation regime. 

To begin with, note that FDR consistency of this example is equivalent to the following problem: for any $\epsilon > 0$, there exists $\delta > 0$ such that for any nominal level $\alpha < \delta$, the BH procedure gives $\fdr \le \epsilon$ under any dependence element in $\bm{\mathcal{P}}$. Let $X_1, \ldots, X_n$ be $n$ standard normal variables that are $\rho$-equicorrelated for some $-\frac1{n-1} \le \rho < 0$. Denote by
\[
Y_i = \sqrt{1 + \rho} X_i + \sqrt{-\rho} Z,
\]
where $Z$ is a standard normal variable independent of $X_1, \ldots, X_n$. As is clear, $Y_1, \ldots, Y_n$ are \iid~standard normal variables. Let $\mathcal{R}_{\alpha, X}$ denote the event that the Simes method at level $\alpha$ rejects the global null hypothesis on $X_1, \ldots, X_n$. Define $\mathcal{R}_{\alpha, Y}$ similarly for $Y_1, \ldots, Y_n$. By definition, in the event $\mathcal{R}_{\alpha, X}$ there must exist some $j \ge 1$ such that
\[
\min \left\{ X_{i_1}, \ldots, X_{i_j}  \right\} \ge \Phi^{-1}\left( 1 - \frac{\alpha j}{n} \right)
\]
for $j$ indices $i_1, \ldots, i_j$ from $\{1, 2, \ldots, n\}$. Let $\mathcal{A}$ denote the event that $Z \ge -\sqrt[4]{n-1}$. Assuming $\alpha < \frac12$, we have
\[
\begin{aligned}
\min \left\{ Y_{i_1}, \ldots, Y_{i_j} \right\} &\ge \sqrt{1 + \rho}\Phi^{-1}\left( 1 - \frac{\alpha j}{n} \right) - \sqrt{-\rho} \sqrt[4]{n-1}\\
&\ge \sqrt{1 - 1/(n-1)}\Phi^{-1}\left( 1 - \frac{\alpha j}{n} \right) - \sqrt{1/(n-1)} \sqrt[4]{n-1}\\
&\ge \frac{n-2}{n-1}\Phi^{-1}\left( 1 - \frac{\alpha j}{n} \right) - (n-1)^{-\frac14}.
\end{aligned}
\]
Now, suppose we have
\begin{equation}\label{eq:phi_ine}
\frac{n-2}{n-1}\Phi^{-1}\left( 1 - \frac{\alpha j}{n} \right) - (n-1)^{-\frac14} \ge \Phi^{-1}\left( 1 - \frac{\alpha_{\epsilon} j}{n} \right)
\end{equation}
for some $\alpha_{\epsilon} \in (0, 1)$ that will be specified later. Then, we get
\[
\min\left\{Y_{i_1}, \ldots, Y_{i_j} \right\} \ge \Phi^{-1}\left( 1 - \frac{\alpha_{\epsilon} j}{n} \right)
\]
in the event $\mathcal{R}_{\alpha, X} \cap \mathcal{A}$. As such, $\mathcal{R}_{\alpha, X} \cap \mathcal{A} \subset \mathcal{R}_{\alpha_{\epsilon}, Y}$. Consequently, we have
\[
\begin{aligned}
\P(\mathcal{R}_{\alpha, X}) &= \P(\mathcal{R}_{\alpha, X} \cap \mathcal{A}) + \P(\mathcal{R}_{\alpha, X} \cap \mathcal{A}^c) \\
& \le \P(\mathcal{R}_{\alpha_{\epsilon}, Y}) + \P(\mathcal{B}^c)\\
& = \P(\mathcal{R}_{\alpha_{\epsilon}, Y}) + \Phi\left(-\sqrt[4]{n-1} \right).
\end{aligned}
\]
From \cite{prds}, it is known that $\P(\mathcal{R}_{\alpha_{\epsilon}, Y}) = \alpha_{\epsilon}$. Therefore, the FDR of the BH procedure applied to the $p$-values from $X_1, \ldots, X_n$ satisfies
\begin{equation}\label{eq:fdr_bund_tosee}
\fdr = \P(\mathcal{R}_{\alpha, X}) \le \alpha_{\epsilon} + \Phi\left(-\sqrt[4]{n-1} \right).
\end{equation}

To continue the proof, we need the following lemma.
\begin{lemma}\label{lm:complex}
For any $\epsilon > 0$, there exists a sufficiently large integer $N$ such that
\begin{equation}\label{eq:phi_ine2}
\frac{n-2}{n-1}\Phi^{-1}\left( 1 - \frac{\alpha j}{n} \right) - (n-1)^{-\frac14} \ge \Phi^{-1}\left( 1 - \frac{\epsilon j}{2n} \right)
\end{equation}
for all $n \ge N$ and all $1 \le j \le n$ if $\alpha < \epsilon/4$.

\end{lemma}

By Lemma~\ref{lm:complex}, we can always find sufficiently large $N_1$, depending only on $\epsilon$, such that \eqref{eq:phi_ine} holds for $\alpha_{\epsilon} = \epsilon/2$ and $\alpha < \epsilon/4$. In this case, we have
\[
\P(\mathcal{R}_{\alpha, X}) \le \frac{\epsilon}{2} + \Phi\left(-\sqrt[4]{n-1} \right).
\]
Next, we choose $N_2$ such that $\Phi\left(-\sqrt[4]{n-1} \right) \le \frac{\epsilon}{2}$ if $n \ge N_2$. Note that $N_2$ depends only on $\epsilon$. Thus, if $n \ge \max\{N_1, N_2\}$, from \eqref{eq:fdr_bund_tosee} we get
\begin{equation}\label{eq:fdr_bund_cdsfsdfs}
\fdr = \P(\mathcal{R}_{\alpha, X}) \le \epsilon
\end{equation}
for any $\alpha < \epsilon/4$. Last, let $\delta_1$ be sufficiently small such that
\begin{equation}\label{eq:fdr_last_cont}
\sup_{\alpha < \delta_1}\fdr(\alpha) \le \epsilon
\end{equation}
for all $n < \max\{N_1, N_2\}$ and $-\frac1{n-1} \le \rho < 0$.

Taken together, \eqref{eq:fdr_bund_cdsfsdfs} and \eqref{eq:fdr_last_cont} show that $\fdr \le \epsilon$ for all $n$ and any $-\frac1{n-1} \le \rho < 0$ if $\alpha < \min \{\epsilon/4, \delta_1\}$. This completes the proof. The proof of Lemma~\ref{lm:complex} is deferred to the Appendix.

\end{proof}

\begin{example}[Arbitrary dependence with vanishing true null proportion] Consider a sequence of integer pairs $\{n^{(l)}, n_0^{(l)}\}_{l=1}^{\infty}$ such that\footnote{Below, the dependence on $l$ is often omitted for the sake of simplicity.} $n^{(l)} \ge n_0^{(l)}$ and
\begin{equation}\label{eq:not_crazy_l}
\sup_l \frac{n_0 \log n_0}{n} < \infty.
\end{equation}
Let $\bm{\mathcal{P}}$ consist of all $n^{(l)}$-dimensional $p$-value distributions with $n^{(l)}_0$ null components. Except for specified $n^{(l)}, n_0^{(l)}$, the null dependence and null-non-null dependence can be arbitrary. Then, $\bm{\mathcal{P}}$ is FDR-consistent.
\end{example}

\begin{proof}
We give a direct proof of the FDR consistency of $\bm{\mathcal{P}}$. By Theorem~\ref{thm:prds_better} in Section~\ref{sec:impr-log-corr}, it suffices to show that
\[
\inf_l \frac1{\pi_0 S(n_0)} > 0
\]
and
\[
\lim_{\alpha \goto 0} S(n_0) \pi_0 \alpha \log\frac{\e}{S(n_0) \pi_0 \alpha} = 0
\]
uniformly over $P \in \bm{\mathcal{P}}$. The first inequality is used to pick a sufficiently small $\alpha$ such that the FDR bound given in Theorem~\ref{thm:prds_better} is nontrivial for all $l$. To verify the two requirements, note that
\[
\frac1{\pi_0 S(n_0)} = (1 + o(1))\frac{n}{n_0 \log n_0} > c
\]
uniformly in $l$ for some $c > 0$, as implied by \eqref{eq:not_crazy_l}. Likewise, we have
\[
S(n_0) \pi_0 \alpha \log\frac{\e}{S(n_0) \pi_0 \alpha} \le \frac{\alpha}{c} \log \frac{\e c}{\alpha},
\]
which tends to zero uniformly as $\alpha \goto 0$.
\end{proof}

%%%%%%%%%% yet another example
\begin{example}[Two-sided-PRDN dependence] \label{ex:two_prdn_s}
Going back to the example of one-sided normal $p$-values given right below Definition~\ref{def:prdn}, observe that the PRDN property is not satisfied by the two-sided $p$-values, unless under certain additional assumptions \cite{karlin1981total}. In fact, as with the PRDS property, PRDN in generally does carry over from one-sided to two-sided $p$-values (see, for example, one-sided $p$-values for Studentized tests in \cite{prds}). Interestingly, these two-sided $p$-values maintain FDR consistency. To state precisely, let $p^{(\text{o})}_1, \ldots, p^{(\text{o})}_n$ be (one-sided) PRDN $p$-values and, in addition, the corresponding test statistics are distributed symmetrically about 0 under the null and are continuous, thereby yielding uniformly distributed null $p$-values. Define the induced two-sided $p$-values as
\[
p^{(\text{t})}_i =
\begin{cases}
2 p^{(\text{o})}_i, & \quad \text{if } p^{(\text{o})}_i  \le \frac12\\
2 (1 - p^{(\text{o})}_i), & \quad \text{if } p^{(\text{o})}_i  > \frac12.
\end{cases}
\]
In the case of normal distribution, the two-sided $p$-value $p^{(\text{t})}_i = 2\Phi(-|z|)$ is given as above by the one-sided $p^{(\text{o})}_i = \Phi(-z)$. Below, we show that all distributions of these two-sided $p$-values form an FDR-consistent dependence class.

\end{example}

\begin{proof}
As earlier, Theorem~\ref{thm:consist} allows us to only consider $p^{(\text{o})}_1, \ldots, p^{(\text{o})}_n$ that satisfy PRD under the global null.
To start with, recognize that the BH procedure at level $\alpha$ controls the FDR at $\alpha$ on these $p$-values. Due to symmetry, this also applies to $1 - p^{(\text{o})}_1, \ldots, 1 - p^{(\text{o})}_n$. Denote by $R$ the number of rejected two-sided $p$-values by the BH procedure. In the event that $R \ge 1$, it must be the case where at least $\ceil{R/2}$ of them are given as $2p^{(\text{o})}_i$ or the case where at least $\ceil{R/2}$ are given as $2(1 - p^{(\text{o})}_i)$. In the former case, at least $\ceil{R/2}$ from $p^{(\text{o})}_1, \ldots, p^{(\text{o})}_n$ are no greater than
\begin{equation}\label{eq:left_o_prds}
\frac 12 \cdot \frac{\alpha R}{n} \le \frac 12 \cdot \frac{2\alpha \ceil{R/2}}{n} = \frac{\alpha \ceil{R/2}}{n}.
\end{equation}
Otherwise, at least $\ceil{R/2}$ from $p^{(\text{o})}_1, \ldots, p^{(\text{o})}_n$ satisfies
\begin{equation}\label{eq:right_o_prds}
\max \left\{ 1 - p^{(\text{o})}_{i_1}, \ldots, 1 - p^{(\text{o})}_{i_{\ceil{R/2}}} \right\} \le  \frac{\alpha \ceil{R/2}}{n}.
\end{equation}
As noted earlier, each of the two events \eqref{eq:left_o_prds} and \eqref{eq:right_o_prds} happens with probability no more than $\alpha$. Taking a union bound, we prove that the FDR on the two-sided $p$-values is controlled at level $2\alpha$. Therefore, it is FDR-consistent. 

\end{proof}

%%%%%%%%%%%%%%%%%%%% last example

\begin{example}[Block dependence]
Fix a positive integer $b$. Consider a $b$-block structure $\{1, 2, \ldots, n\} = \cup_{i=1}^m B_i$, each of size $|B_i| = b_i \le b$ and $b_1 + \cdots + b_m = n$. Let $\bm{\mathcal{P}}$ be the class of all $p$-value distributions with a $b$-block structure in the sense that $p$-values from different blocks are jointly independent while, within every block, the joint distribution can be arbitrary. This dependence class is FDR-consistent.
\end{example}

\begin{proof}
Consider the induced null set, which remains block-structured. The induced block sizes are no more than $b$. Thus, by Theorem~\ref{thm:consist}, it suffices to consider the global null case.

Denote the $p$-values by $p_1^{(l)}, p_2^{(l)}, \ldots, p_{b_l}^{(l)}$ for the $l$th block, where $l = 1, \ldots, m$. Now, we consider the adjusted $p$-value
\[
p^{(l)} := \min\left\{ b_l \min_{i} p_i^{(l)}, 1 \right\}.
\]
Note that these are valid $p$-values. Now, in the case where the BH procedure rejects at least one $p$-value, that is, $R > 0$, we know that at least $\ceil{R/b}$ blocks are rejected. By assumption, the minimum $p$-values from at least $\ceil{R/b}$ blocks are all less than or equal to
\[
\frac{\alpha R}{n} \le \frac{b\alpha \ceil{R/b}}{n} \le \frac{b\alpha \ceil{R/b}}{m}.
\]
Thus, whenever the BH procedure rejects at least one on the full set of $p$-values, it also rejects at least one when applied to the adjusted $p$-values $p^{(1)}, \ldots, p^{(m)}$. Thus, due to the joint independence of the $m$ blocks, we must have
\[
\fdr \le b \alpha,
\]
which tends to 0 uniformly over the dependence class. Thus, this dependence class is FDR-consistent.

\end{proof}

%%% Local Variables:
%%% mode: latex
%%% TeX-master: "paper"
%%% End:

\section{FDR Control Under Arbitrary Dependence}
\label{sec:impr-log-corr}

Under arbitrary dependence, \cite{prds} proves that the FDR of the BH procedure always satisfies
\begin{equation}\label{eq:log_corre}
\fdr \le \min \left\{ S(n) \pi_0 \alpha, 1 \right\},
\end{equation}
where $S(n)$ is a shorthand for $1 + \frac12 + \cdots + \frac1n \approx \log n + 0.577$. In the literature, to our knowledge, this $\log$-correction bound is the best known unconditional FDR bound for the full range $0 < \alpha < 1$. As a matter of fact, if $(1 - \pi_0 + \pi_0 S(n_0))\alpha \le 1$, this FDR bound is tight in the sense that there exists a joint distribution of the $p$-values such that \eqref{eq:log_corre} is an equality \cite{guo2008control} (see also \cite{rodland2006simes} for the case of $\pi_0 = 1$). Recognizing this fact, one would imagine that it is very challenging to improve on the $\log$-correction bound. 

Interestingly, a direct application of the \texttt{FDR-linking} theorem allows us to strictly improve the bound \eqref{eq:log_corre} for a certain range of the nominal level $\alpha$.

\begin{theorem}\label{thm:prds_better}
Under arbitrary dependence of the $p$-values, any compliant procedure at level $\alpha$ satisfies
\[
\fdr \le 
\begin{cases}
1, & \text{if } \alpha \ge \frac1{\pi_0 S(n_0)}\\
 S(n_0) \pi_0 \alpha \log\frac{\e}{S(n_0) \pi_0 \alpha}, & \text{if } \alpha < \frac1{\pi_0 S(n_0)}.
\end{cases}
\]
\end{theorem}

This theorem applies to all compliant procedures, whereas it is unclear whether the bound \eqref{eq:log_corre} applies to general compliant procedures. To further appreciate this theorem, we take
\[
\alpha = \frac1{\pi_0 S(n)}.
\]
With this choice of $\alpha$, \eqref{eq:log_corre} yields the trivial bound $\fdr \le 1$, whereas Theorem~\ref{thm:prds_better} gives
\[
\fdr \le \frac{S(n_0)}{S(n)} \log\frac{\e S(n)}{S(n_0)} < 1
\]
under the very mild condition $n > n_0$ (that is, any configurations of the $p$-values except for the global null). More broadly, the new unconditional FDR bound strictly improves \eqref{eq:log_corre} for the following range of $\alpha$:
\begin{equation}\label{eq:n_n_0_b}
\frac1{\pi_0 S(n_0)} \e^{1 - \frac{S(n)}{S(n_0)}}< \alpha < \frac1{\pi_0 S(n_0)}.
\end{equation}
To see this point, note that it suffices to verify
\[
S(n_0) \pi_0 \alpha \log\frac{\e}{S(n_0) \pi_0 \alpha} < S(n) \pi_0 \alpha
\]
if $\alpha$ satisfies \eqref{eq:n_n_0_b}. This is clearly true.

We conclude this section by proving Theorem~\ref{thm:prds_better}.
\begin{proof}[Proof of Theorem~\ref{thm:prds_better}]
Using the unconditional FDR bound \cite{hommel1983tests,prds}
\[
\fdr_0(x) \le \min\left\{S(n_0) x, 1\right\}
\]
under arbitrary dependence of the null $p$-values, from
Theorem~\ref{thm:general} we get
\[
\fdr \le \pi_0\alpha + \pi_0 \alpha\int^{1}_{\pi_0 \alpha} \frac{\min\left\{S(n_0) x, 1\right\}}{x^2} {\dx}x.
\]
This FDR bound can be simplified as
\[
\pi_0\alpha + \pi_0 \alpha\int^{1}_{\pi_0 \alpha} \frac{\min\left\{S(n_0) x, 1\right\}}{x^2} {\dx}x = 
\begin{cases}
1, & \text{if } \alpha \ge \frac1{\pi_0 S(n_0)}\\
 S(n_0) \pi_0 \alpha \log\frac{\e}{S(n_0) \pi_0 \alpha}, & \text{if } \alpha < \frac1{\pi_0 S(n_0)}.
\end{cases}
\]
This concludes the proof.

\end{proof}

%%% Local Variables:
%%% mode: latex
%%% TeX-master: "paper"
%%% End:

\section{Bounded Adversariness}
\label{sec:bound-advers}

As shown in Theorem~\ref{thm:optimal}, the logarithmic factor $\log\frac1\alpha$ in the FDR bound is unavoidable provided the informed adversarial non-null $p$-values. A question arising from this observation is whether the logarithmic factor can be removed if the adversary is restricted from accessing certain information of the null $p$-values.

Going back to the informal explanation of the factor $\log\frac1\alpha$ surrounding \eqref{eq:adver_informal}, we learn that the smallest null $p$-value is essential for the adversary to construct non-null $p$-values that maximize the FDP. In light of this observation, we consider a \textit{Bonferroni-masked} adversary that has access to all null $p$-values but the smallest. Thus, the available information is not sufficient to perform the Bonferroni correction on the nulls (unless $p_{(2)}^{\textnormal{null}}$ is below the Bonferroni threshold), hence the name Bonferroni-masked adversary. To formally define this adversary, we use a function
\begin{equation}\nonumber
T = T(p_{(2)}^{\textnormal{null}}, \ldots, p_{(n_0)}^{\textnormal{null}})
\end{equation}
to denote the number of true rejections made by the BH procedure, which depends on all sorted null $p$-values except for the smallest one. This can be realized by setting $T$ non-null $p$-values to 0 and the rest to 1. 

The following theorem shows that the logarithmic factor can be dropped in the FDR bound with independent nulls and any Bonferroni-masked non-nulls. We assume $n_0 \ge 2$ to exclude trivial cases.
\begin{theorem}\label{thm:no_log_term}
Let the null $p$-values be \iid~uniform variables on $(0, 1)$ and the non-null $p$-values be provided by any Bonferroni-masked adversary. Then, any compliant procedure at level $\alpha$ ensures
\[
\fdr \le 3.5\alpha.
\]
\end{theorem}

As the proof of Theorem~\ref{thm:optimal} also assumes \iid~uniform null $p$-values, the improvement of the FDR bound in Theorem~\ref{thm:no_log_term} can only be attributed to the withholding of the most significant null $p$-value to the Bonferroni-masked adversary. As an aside, we remark that the independence of the null $p$-values is crucial to Theorem~\ref{thm:no_log_term} since, in the presence of correlation, other null $p$-values would reveal much information about the smallest null $p$-value.

Now, we prove Theorem~\ref{thm:no_log_term} as follows. 
\begin{proof}[Proof of Theorem~\ref{thm:no_log_term}]

As earlier, $V$ denotes the number of false rejections. If $V \ge 2$, we have
\[
\fdp \le \max_{2 \le j \le n_0} \frac{\alpha j}{n p^{\textnormal{null}}_{(j)}},
\]
which remains true if $V = 0$ because $\fdp = 0$ in this case. In the case of $V = 1$, the compliance condition shows that the rejected null $p$-value is below $\alpha(T+1)/n$, thus implying
\[
\pnull_{(1)} \le \frac{ \alpha(T+1)}{n}.
\]
Taken together, the two inequalities above give
\begin{equation}\label{eq:p_t_1_?}
\begin{aligned}
\fdr &= \E (\fdp; R \ne T+1) + \E (\fdp; R = T+1)\\
&\le \E \left[ \max_{2 \le j \le n_0} \frac{\alpha j}{n p^{\textnormal{null}}_{(j)}}; R \ne T+1 \right] + \E \left[\frac1{T+1}; R = T+1 \right]\\
&\le \E \left[ \max_{2 \le j \le n_0} \frac{\alpha j}{n p^{\textnormal{null}}_{(j)}}\right] + \E \left[\frac{\bm{1}\left( \pnull_{(1)} \le \alpha(T+1)/n \right)}{T+1}\right].
\end{aligned}
\end{equation}

To bound the term contributed by the case of $R \ne T + 1$ in \eqref{eq:p_t_1_?},  note that Lemma 3.1 of \cite{dwork2018differentially} ensures that
\begin{equation}\label{eq:first_c_pi}
\E \left[ \max_{2 \le j \le n_0} \frac{\alpha j}{n p^{\textnormal{null}}_{(j)}}\right] \le C_2 \pi_0 \alpha
\end{equation}
for a constant $C_2$ between $2.4$ and $2.5$. Now we turn to the term contributed by the case of $R = T+1$. First, write
\[
\E \left[\frac{\bm{1}\left( \pnull_{(1)} \le \alpha(T+1)/n \right)}{T+1}\right] = \E \left\{\E\left[\frac{\bm{1}\left( \pnull_{(1)} \le \alpha(T+1)/n \right)}{T+1} \Bigg| \pnull_{(2)}, \ldots, \pnull_{(n_0)}\right] \right\}.
\]
Recognizing that $\pnull_{(1)}$ is (conditionally) uniformly distributed on $(0, \pnull_{(2)})$ and $T$ is measurable with respect to $\pnull_{(2)}, \ldots, \pnull_{(n_0)}$, we get
\[
\E\left[\frac{\bm{1}\left( \pnull_{(1)} \le \alpha(T+1)/n \right)}{T+1} \Bigg| \pnull_{(2)}, \ldots, \pnull_{(n_0)}\right] \le \frac{\frac{\alpha(T+1)/n}{\pnull_{(2)}}}{T+1} = \frac{\alpha}{n \pnull_{(2)}},
\]
where the inequality reduces to equality if $\alpha(T+1)/n \le \pnull_{(2)}$. This yields
\begin{equation}\label{eq:first_c_pi_22}
\E \left[\frac{\bm{1}\left( \pnull_{(1)} \le \alpha(T+1)/n \right)}{T+1}\right] \le \E \left[ \frac{\alpha}{n \pnull_{(2)}} \right].
\end{equation}
To proceed, note the fact that $\pnull_{(2)}$ follows the Beta distribution parametrized by $(2, n_0 -1)$. By means of this fact, we have
\[
\E \left[ \frac{\alpha}{n \pnull_{(2)}} \right] = \frac{\alpha}{n} \int_0^1 \frac1x \frac{x (1-x)^{n_0-2}}{\mathrm{Beta}(2, n_0 - 1)} {\dx}x= \pi_0 \alpha,
\]
which, together with \eqref{eq:first_c_pi} and \eqref{eq:first_c_pi_22}, reveals that
\[
\fdr \le C_2 \pi_0 \alpha + \pi_0 \alpha < 3.5 \pi_0 \alpha \le 3.5 \alpha.
\]
This completes the proof.

\end{proof}

%%% Local Variables:
%%% mode: latex
%%% TeX-master: "paper"
%%% End:

\section{Discussion}
\label{sec:discussion}

In this paper, we have demonstrated the use of the \texttt{FDR-linking} theorem in proving FDR control properties through examples for compliant procedures, including the BH procedure. In slightly more detail, this theorem shows that any compliant procedures control the FDR up to a factor that is independent of the number of hypotheses under the PRDN dependence of the $p$-values, which strictly relaxes the PRDS property by only imposing distributional assumptions on the nulls. Moreover, the FDR bound is optimal under this new dependence, shedding light on the sharpness of the \texttt{FDR-linking} theorem. In another application, we propose FDR consistency as a flexible, amenable concept to complement FDR control and, using the \texttt{FDR-linking} theorem, prove that this new property relies entirely on the null $p$-values, regardless of the non-nulls.

Loosely speaking, the \texttt{FDR-linking} theorem reveals that FDR control is basically a matter of the null $p$-values. This is a blessing for researchers in that an analysis of the FDR would not lose much of its value by focusing only on the joint distribution of the null $p$-values. This is particularly convenient if some knowledge of the null dependence is available. From a different angle, this theorem unveils the robustness of the BH procedure in FDR control against any unfavorable or even adversarial dependence between the null $p$-values and non-null $p$-values. This viewpoint is in agreement with the widely observed phenomenon that the BH procedure seldom loses much control over the FDR beyond what has been proved \cite{storey2003positive,reiner2007fdr,ge2008some,clarke2009robustness}.

We conclude this paper with several promising directions for future work. First, the simple bound on the FDP \eqref{eq:fdp_main0} can possibly be extended to variants of the FDR criterion and general BH-type multiple testing procedures. For instance, the false discovery exceedance \cite{genovese2004stochastic,van2004augmentation,romano2007,guo2014further}, defined as the probability that the FDP exceeds a specified level, is a popular alternative to the FDR since the realized FDP can significantly deviate from its expectation (FDR) in certain settings~\cite{efron2007correlation}. To control the false discovery exceedance, it follows from \eqref{eq:fdp_main0} that
\[
\P(\fdp \ge \gamma) \le \P\left(  \max_{1 \le j \le n_0} \frac{\alpha j}{n p^{\textnormal{null}}_{(j)}}  \ge \gamma \right) \le \frac{\pi_0\alpha}{\gamma}
\]
for any $0 < \gamma < 1$ if the $p$-values satisfy the PRDN property. Moreover, the bound \eqref{eq:fdp_main0} can be used to derive a bound on $\E \fdp^k$ for an integer $k \ge 2$ \cite{ferreira2006benjamini} and to possibly maximize the variance of FDP, thereby delineating the variability of FDR control. It is also of interest to incorporate prior knowledge such as weights of hypotheses and the true null proportion into the FDP bound~\cite{genovese2006false,benjamini2006adaptive,storey2004strong,storey2002direct,ramdas2017unified}.

More broadly, the \texttt{FDR-linking} inequality \eqref{eq:fdr_link} in Theorem~\ref{thm:general} calls for further investigation. In addition to Theorem~\ref{thm:optimal}, we wonder if the tightness of \eqref{eq:fdr_link} remains under general dependence of the nulls. Practically speaking, the non-null $p$-values are unlikely to be the least favorable for FDR control as in Theorem~\ref{thm:optimal} and, therefore, it is worth trying to improve the FDR bound in the \texttt{FDR-linking} theorem by assuming more realistic conditions on the null-non-null dependence. Theorem~\ref{thm:no_log_term} has made a step toward this goal. Due to the connection between $\fdr_0(x)$ and the Simes method, a related direction is to explore null dependence under which the Simes method controls the type I error at a reasonably small level. For example, it is likely that for a $\rho$-equicorrelated multivariate normal distribution, the Simes $p$-value would be quite conservative for a range of $\rho$. If so, it would lead to an FDR bound of the form $c \alpha$ for some constant $c$
using the \texttt{FDR-linking} theorem.

Finally, we present a challenging problem whether FDR consistency holds for many common distributions of $p$-values. Specifically, we ask:\\[0.5em]
{\noindent\bf Open Problem.} {\em Prove or disprove that $p$-values for testing means in any multivariate normal distribution with known covariance form an FDR-consistent dependence class.}\\[0.5em]
By Theorem~\ref{thm:consist}, we only need to focus on the global null case, and it does not matter whether the $p$-values are two-sided or one-sided as in Example~\ref{ex:equi_one_two}. We believe that this problem is true given that the Simes method is empirically observed to be approximately valid under many types of dependence structures \cite{reiner2007fdr}, though pathological counterexamples exist \cite{rodland2006simes,finner2017simes}.

%%% Local Variables:
%%% mode: latex
%%% TeX-master: "paper"
%%% End:

\subsection*{Acknowledgments}

WJS is grateful for helpful feedback from the audience of the workshops ``Adaptive Data Analysis'' and ``Robust and High-Dimensional Statistics'' held at the Simons Institute for the Theory of Computing at Berkeley. This work was supported in part by the NSF via grant CCF-1763314.

{\small
\bibliographystyle{alpha}
\bibliography{ref}
}

\clearpage
\appendix
\section{Appendix}
%{Proof of Technical Results}
\label{sec:proof-techn-results}

\subsection{The PRDN example in Section~\ref{sec:introduction}}
\label{sec:proof-section-ref}

The one-sided $p$-values given as an example right below Definition~\ref{def:prdn} satisfy the PRDN property. The simple proof below follows from that of Case 1 in Section 3.1 of \cite{prds}. 

We introduce some notation for the proof. Let $X_{\N_0}$ be the vector comprised of the $n_0$ null components of $X$. For any $i \in \N_0$, let $X_{\N_0}^{(i)}$ denote the entry corresponding to $i$ and $X_{\N_0}^{(-i)}$ denote the vector derived from $X_{\N_0}$ by deleting $X_{\N_0}^{(i)}$. Let $\Sigma^0_{-i,-i}$ be the covariance matrix of $X_{\N_0}^{(-i)}$, $\Sigma^0_{i,-i}$ be the covariance of $X_{\N_0}^{(i)}$ and $X_{\N_0}^{(-i)}$, and last, $\Sigma^0_{i,i}$ be the variance of $X_{\N_0}^{(i)}$. In this setup, the distribution of $X_{\N_0}^{(-i)}$ given $X_{\N_0}^{(i)} = x$ is a normal distribution with mean and covariance given as
\[
\Sigma^0_{i,-i}(\Sigma^0_{i,i})^{-1} x, ~ \Sigma^0_{-i,-i} - \Sigma^0_{i,-i} (\Sigma^0_{i,i})^{-1} (\Sigma^0_{i,-i})^\top,
\]
respectively. 

Above, while the conditional covariance $\Sigma^0_{-i,-i} - \Sigma^0_{i,-i} (\Sigma^0_{i,i})^{-1} (\Sigma^0_{i,-i})^\top$ is fixed, the conditional mean $\Sigma^0_{i,-i}(\Sigma^0_{i,i})^{-1} x$ increases with $x$ since $\Sigma^0_{i,-i}$ is entrywise nonnegative. This fact, together with the observation that the one-sided normal $p$-value is a monotone transformation of the $z$-value, implies that the one-sided $p$-values satisfy the PRDN property.

For information, we give a sufficient condition for the two-sided normal $p$-values normal to satisfy the PRDN property. The setting is the same as above. To this end, we only need to focus on the covariance $\Sigma^0$ of the null part $X_{\N_0}$ of the full vector $X$. As shown in \cite{karlin1981total}, the density of $|X_{\N_0}|$ is $\text{MTP}_2$ if and only if there exists a diagonal matrix $B$ with diagonal entries $\pm 1$ such that the off-diagonal entries of $-B (\Sigma^0)^{-1}B$ are all nonnegative. Thus, the existence of such a matrix $B$ ensures that the two-sided $p$-values satisfy the PRDN property because the $\text{MTP}_2$ condition implies PRDS on any subset. As an aside, this sufficient condition might not be necessary.

\subsection{Proof of Lemma~\ref{lm:complex} in Section~\ref{sec:colt-style}}
\label{sec:proofs-forxxxx}

%%Let $\delta_2 = \epsilon/4$. 

Below, we prove Lemma~\ref{lm:complex}.

\begin{proof}[Proof of Lemma~\ref{lm:complex}]
The lemma is equivalent to
\begin{equation}\label{eq:lm_eqv}
\Phi^{-1}\left( 1 - \frac{\epsilon j}{4n} \right) - \Phi^{-1}\left( 1 - \frac{\epsilon j}{2n} \right) \ge \frac1{n-1}\Phi^{-1}\left( 1 - \frac{\epsilon j}{4n} \right) + (n-1)^{-\frac14}.
\end{equation}
On the one hand, for sufficiently large $n$ and any $1 \le j \le n$, we know
\begin{equation}\label{eq:rd_qqq}
\begin{aligned}
\frac1{n-1}\Phi^{-1}\left( 1 - \frac{\epsilon j}{4n} \right) + (n-1)^{-\frac14} &= \frac{O(1)\sqrt{\log(4n/(\epsilon j))}}{n-1} + (n-1)^{-\frac14}\\
& = O(n^{-\frac14}).
\end{aligned}
\end{equation}
On the other hand, we have
\begin{equation}\label{eq:quarter_rd}
\begin{aligned}
\Phi^{-1}\left( 1 - \frac{\epsilon j}{4n} \right) - \Phi^{-1}\left( 1 - \frac{\epsilon j}{2n} \right) &\gtrsim \frac{2\log 2 }{2\sqrt{2\log \frac{4n}{\epsilon j}}} \\
&\gtrsim \frac1{\sqrt{\log n}}
\end{aligned}
\end{equation}
if $n$ is sufficiently large. Combining \eqref{eq:rd_qqq} and \eqref{eq:quarter_rd} proves \eqref{eq:lm_eqv} if $n \ge N$ for some $N$ that only depends on $\epsilon$.

\end{proof}

%%% Local Variables:
%%% mode: latex
%%% TeX-master: "paper"
%%% End:

\end{document}